\documentclass[letterpaper,11pt]{amsart}
\usepackage{tikz}
\usepackage{wrapfig}
\usetikzlibrary{arrows}
\usepackage{float}
\usepackage{subcaption} 
\usepackage{amsmath}
\usepackage{forest}
\usepackage[font=footnotesize,labelfont=bf]{caption}
\usepackage{latexsym,array,delarray,amsthm,amssymb,epsfig,setspace}%eufrak

\usepackage{mathtools} % for "\DeclarePairedDelimiter" macro

\theoremstyle{plain}
\newtheorem{thm}{Theorem}[section]
\newtheorem{lemma}[thm]{Lemma}
\newtheorem{prop}[thm]{Proposition}
\newtheorem{cor}[thm]{Corollary}
\newtheorem{conj}[thm]{Conjecture}
\newtheorem{alg}[thm]{Algorithm}
\newtheorem*{thm*}{Theorem}
\newtheorem*{lemma*}{Lemma}
\newtheorem*{prop*}{Proposition}
\newtheorem*{cor*}{Corollary}
\newtheorem*{conj*}{Conjecture}
\newtheorem*{alg*}{Algorithm}

\theoremstyle{definition}
\newtheorem{defn}[thm]{Definition}
\newtheorem{ex}[thm]{Example}

\theoremstyle{remark}
\newtheorem*{rmk}{Remark}

\newcommand{\zz}{\mathbb{Z}}

\newcommand{\pp}{\mathbb{P}}

\newcommand{\rr}{\mathbb{R}}
\newcommand{\cc}{\mathbb{C}}

\newcommand{\ii}{\mathcal{I}}
\newcommand{\mm}{\mathcal{M}}

\newcommand{\ind}{\mbox{$\perp \kern-5.5pt \perp$}}

\newcommand{\zeroideal}{\langle 0 \rangle}
\DeclareMathOperator{\rank}{rank}

\title{Identifiability in Phylogenetics using Algebraic Matroids}
\author{Benjamin Hollering and Seth Sullivant}

\begin{document}
\maketitle

\begin{abstract}
Identifiability is a crucial property for a statistical model since distributions in the model uniquely determine the parameters that produce them. In phylogenetics, the identifiability of the tree parameter is of particular interest since it means that phylogenetic models can be used to infer evolutionary histories from data. In this paper we introduce a new computational strategy for proving the identifiability of discrete parameters in algebraic statistical models that uses algebraic matroids naturally associated to the models. We then use this algorithm to prove that the tree parameters are generically identifiable for 2-tree CFN and K3P mixtures. We also show that the $k$-cycle phylogenetic network parameter is identifiable under the K2P and K3P models. 
\end{abstract}

\section{Introduction}
\label{sec:intro}
A statistical model is identifiable if the map parameterizing the model is injective. This means that the parameters producing a probability distribution in the model can be uniquely determined from the distribution itself which is a critical property for meaningful data analysis. In phylogenetic models, the identifiability of the tree parameter is especially important since this allows for evolutionary histories to be inferred from observed genetic data. 

The identifiability of the  tree parameter in basic models has already been established \cite{treeIDChang96} and a natural next step is to investigate the identifiability of the tree parameters in phylogenetic mixture models. Mixture models can be used to represent more complicated evolutionary events such as horizontal gene transfer. 
In \cite{matsenCFNID07}, Matsen and Steel showed that the tree parameters are not 
identifiable for 2-tree mixtures on four leaf trees under the Cavendar-Farris-Neyman (CFN) model.
On the other hand, positive results for the identifiability of tree parameters in other 
group-based models were obtained in both \cite{mixJCID10} and \cite{id3MixJC15}.
In \cite{mixJCID10}, the authors constructed linear invariants for 
2-tree Jukes-Cantor (JC) and Kimura 2-Parameter (K2P) mixtures to show that the tree parameters were identifiable and \cite{id3MixJC15} used direct computation to construct invariants for 3-tree JC mixtures to obtain identifiability results. These computations often involve time consuming 
Gr\"{o}bner basis computations, which are not possible to do for larger models. 
Similar calculations were also done in \cite{phyloNetId18} to establish the 
identifiability of the network parameters in Jukes-Cantor network models.

Our goal in this paper is to introduce a new algorithm that can be used to show that 
parameters of an algebraic statistical model are identifiable by computing 
independent sets in a naturally associated algebraic matroid. This allows us to 
avoid dealing with the vanishing ideals that are typically used and thus avoid 
Gr\"{o}bner basis calculations. We begin with a short background on generic 
identifiability and algebraic matroids in Section \ref{sec:prelim}. 
We then introduce the main algorithm we employ to prove identifiability results in 
Section \ref{sec:algID}.  We provide both an exact verification based on symbolic
computation and a randomized algorithm with probabilistic guarantees based on the Schwartz-Zippel Lemma.
In Section \ref{sec:phylo} we use the algorithm and
the Six-To-Infinity Theorem \cite{mixedUpTrees08} to show that the tree parameters are generically identifiable in 2-tree CFN and K3P mixture models. We end by showing how our algorithm can be used to extend the results in \cite{phyloNetId18} for JC phylogenetic networks to K2P and K3P networks.

%%%%%%%%%%%%%%%%%%%%%%%%%%%%%%%%%%%%%%%%%%%%%%%%%%%%%%%%%%%%%%
%%%%%%%%%%%%%%%%%%%%%%%%%%%%%%%%%%%%%%%%%%%%%%%%%%%%%%%%%%%%%%
%%%%%%%%%%%%%%%%%%%%%%%%%%%%%%%%%%%%%%%%%%%%%%%%%%%%%%%%%%%%%%
%%%%%%%%%%%%%%%%%%%%%%%%%%%%%%%%%%%%%%%%%%%%%%%%%%%%%%%%%%%%%%

\section{Preliminaries}
\label{sec:prelim}
In this section we provide some background on identifiability and describe some common tools used to prove identifiability results. We then discuss algebraic matroids which will be the main tool we use to prove identifiability results in this paper.

\subsection{Generic Identifiability in Algebraic Statistics}
Our main objects of focus in this paper will be parametric algebraic statistical models for discrete random variables. This means we have a rational map
\[
\phi : \Theta \to \Delta_r = \bigg\{p \in \rr^{r+1} ~:~ \sum_{i = 1}^{r+1} p_i = 1, ~ p_i \geq 0 \mbox{ for all $i$}  \bigg\}
\]
whose image we denote with $M$ is the model which sits inside a probability simplex $\Delta_r$.
This is a broad setting that includes many classic statistical models such as distributions of discrete random variables and the phylogenetic models that we will discuss in the later sections. The definitions and techniques presented in this paper could also be adapted for Gaussian random variable and other continuous models with finite dimensional natural parameter spaces.   

If we have a family of these models $\{M_s\}_{s=1}^k$ that all sit inside $\Delta_r$ and are indexed by a discrete parameter $s$, then we say that the discrete parameter $s$ is \emph{globally identifiable} if
$M_{s_1} \cap M_{s_2} = \emptyset$ for every distinct pair $\{s_1, s_2\}$ of values of $s$. Most models are not globally identifiable but may still satisfy a slightly weaker notion of identifiability instead. 

\begin{defn}
\label{def:genID}
Let $\{M_s \}_{s=1}^k$ be a collection of algebraic models that sit inside the probability simplex $\Delta_r$, then the parameter $s$ is \emph{generically identifiable} if for each 2-subset $\{s_1, s_2 \} \subset [k]$, 
\[
\dim(M_{s_1} \cap M_{s_2}) < \min(\dim(M_{s_1}), \dim(M_{s_2}))
\]
\end{defn}
Another way to think about generic identifiability is that the overlap of any two models in the family is a Lebesgue measure zero subset of both of the overlapping models. A typical tool for proving generic identifiability of algebraic models is the following proposition that uses the \emph{vanishing ideal} $\ii(M) = \{f \in \cc[p] ~:~ f(p) = 0 \mbox{ for all } p \in M\}$ of the model $M$.  

\begin{prop}
\label{prop:idealID}
\cite[Proposition 16.1.12]{algStat18}
Let $M_1$ and $M_2$ be two irreducible algebraic models which sit inside the probability simplex $\Delta_r$. If there exists polynomials $f_1$ and $f_2$ such that
\[
f_1 \in \mathcal{I}(M_1)\setminus \mathcal{I}(M_2) ~ \mbox{and}  ~
f_2 \in \mathcal{I}(M_2)\setminus \mathcal{I}(M_1)
\]
then $\dim(M_{1} \cap M_{2}) < \min(\dim(M_{1}), \dim(M_{2}))$. 
\end{prop}

If the models $M_1$ and $M_2$ have the same dimension, then to ensure their intersection is lower dimensional, it suffices to show that $\ii(M_1) \neq \ii(M_2)$. This means it is enough to find either $f \in \ii(M_1) \setminus \ii(M_2)$ or
$f \in \ii(M_2) \setminus \ii(M_1)$.

We note here that the vanishing ideal of $M$ also completely defines the \emph{Zariski closure} of the model which is the algebraic variety  $\overline{M} = \{p \in \cc^{r+1} ~:~  f(p) = 0 \mbox{ for all } f \in \ii(M) \}$. Essentially, $\ii(M)$ gives an \emph{implicit} description of the model $M$. Computing the implicit description $\ii(M)$ typically requires Gr\"{o}bner basis computations which can be difficult, especially as the number of variables involved increases. 

%%%
%%%
%%%

\subsection{Algebraic Matroids}
In this section we introduce some basic concepts from matroid theory 
and some results on \emph{algebraic matroids} defined by irreducible varieties. 
The results collected here will be the main tools that we utilize to prove identifiability results. 

\begin{defn}\label{def:matroid}
A \emph{matroid} $\mm = (E,\ii)$ is a pair where $E$ is a finite set and $\mathcal{I} \subseteq 2^E$ satisfies
\begin{enumerate}
    \item $\emptyset \in \ii$.
    \item If $I' \subseteq I \in \ii$, 
    then $I' \in \ii$.
    \item If $I_1, I_2 \in \ii$ and $|I_2| > |I_1|$, then there exists $e \in I_2 \setminus I_1$ such that $I_1 \cup e \in \ii$.
\end{enumerate}
\end{defn}

There are many equivalent formulations of the axioms of a matroid but Definition \ref{def:matroid}
will be sufficient for the purpose of this paper (however, see \cite{oxleyMatroidTheory92} 
for further details). 

A classic example is the matroid defined by a matrix which is described more fully in the following example. 

\begin{ex}
\label{ex:linearMatroid}
Let $A \in k^{m \times n}$ be a matrix with entries in a field $k$ and $a_1,\ldots a_n$ be the columns of $A$. Then letting $E = [n]$ and taking $\ii$  to be the subsets of $E$ such that the corresponding columns of $A$ are linearly independent over $k$, defines a matroid. A matroid defined in this way is called a \emph{linear matroid} over the field $k$. More concretely, suppose that
\[
A = \left[ \begin{array}{cccc}
    1 & 1 & -1 & -2  \\
    3 & 1 &  2 & 4  \\
    0 & -1 & 1 & 2  \\
\end{array} \right]. 
\]
and for any $S \subseteq [4]$ let $A_S$ denote the submatrix of $A$ obtained by taking only the columns indexed by $S$. A set $S$ is an independent set in the matroid $\mm(A)$ defined by $A$ if and only if $\rank(A_S) = |S|$. In this case the independent sets of $\mm(A)$ are 
\[
\emptyset, \{1\}, \{2\}, \{3\}, \{4\}, \{1,2\}, \{1,3\}, \{1,4\}, \{2,3\}, \{2,4\}, \{1,2,3\}, \{1,2,4\}.
\]
Note that the maximal independent sets, called the bases of $\mm(A)$, completely determine all of the independent sets. 
\end{ex}

Linear matroids are one of the key examples of matroids, and the name matroid itself is
supposed to indicate that matroids form a generalize of this linear algebraic
independence structure arising from a matrix.  

There is also a way to naturally associate a matroid to a variety 
and some recent work has been done studying and applying such matroids 
\cite{algMatroidsWithGraphSym,compAlgMat14,algMatroidsInAction}. 
All these matroids are examples of \emph{algebraic matroids} though for practical
purposes it can be more useful to think of the following geometric characterization.

\begin{defn}
Let $V \subset k^n$ be an irreducible variety over the field 
$k$ and for $S \subseteq [n]$ let $\pi_S: k^n \to k^{|S|}$ be the projection 
onto the coordinates in $S$.  Let $\overline{\pi_S(V)}$ denote the Zariski closure of the 
projection of $V$.   Then the pair $([n],\ii_V)$ defines a matroid where
\[
\ii_V = \{S \subseteq [n] ~:~ \overline{\pi_S(V)} = k^{|S|} \}
\]
which is called the \emph{coordinate projection matroid} of $V$ and denote by $\mm(V)$. 
\end{defn}

The geometric perspective on algebraic matroids can also be phrased in an algebraic language.

\begin{prop}
\label{prop:idealMat}
Let $V \subset k^n$ be an irreducible variety.  
Let $P \subseteq k[x_1, \ldots, x_n]$ be the 
vanishing ideal of $V$.
A set $S$ is an independent set of the coordinate projection matorid $\mm(V)$
if and only if 
\[
 P \cap k[x_i : i \in S] = \langle 0 \rangle.
\] 
\end{prop}

\begin{proof}
This follows directly from the fact that $P \cap k[x_i : i \in S]$ is the
vanishing ideal of the coordinate projection $\pi_S(V)$ and the fact that
the vanishing ideal of a set is $\langle 0 \rangle$ if and only if its Zariski closure
is all of space.  
\end{proof}

Recall the more familiar definition of an algebraic matroid.

\begin{defn}
Let $L / k$ be a field extension and let $E = \{\alpha_1, \ldots, \alpha_n \} \subseteq L$.
The \emph{algebraic matroid} $(E, \ii)$ consists of all sets $S \subseteq E$ that
are algebraically independent over $k$.  
\end{defn}

Note that Proposition \ref{prop:idealMat}  shows that the coordinate projection matroid
is an algebraic matroid where the field extension is ${\rm Frac}(k[x_1, \ldots, x_n]/P) / k$ and
$E = \{x_1, \ldots, x_n \}$, the images of the variables in the
fraction field ${\rm Frac}(k[x_1, \ldots, x_n]/P)$.

When the variety $V$ is parameterized we are able to construct 
the matroid $\mm(V)$ using the Jacobian matrix of the 
parameterization (see \cite{compAlgMat14}). 

\begin{prop}
Suppose that 
$\phi(\theta_1, \ldots, \theta_d) = (\phi_1(\theta), \ldots, \phi_n(\theta))$ 
parameterizes $V$ (that is, $V = \overline{\phi( k^d )}$). Let
\begin{equation}
\label{eqn:paramJac}
    J(\phi) = \left(\frac{\partial \phi_j}{\partial \theta_i} \right), 1 \leq i \leq d ,~ 1 \leq j \leq n 
\end{equation}
be the transpose of the Jacobian matrix of $\phi$. 
Then the matroid defined by the matrix $J(\phi)$ using linear independence 
over the fraction field $\mbox{Frac}\left( k[\mathbf{\theta}] \right) = k(\theta)$
gives the same matroid as $\mm(\overline{\phi(k^d)})$.  
\end{prop}

Thus we have multiple different ways that we can view the same 
matroid which will be convenient to use at different times. 
We end this section with an example that illustrates these different versions of the same matroid. 

\begin{ex}\label{ex:binomialrv}
Let $M \subset \pp^2$ be the model for a binomial random variable with 2 trials in projective space. This model is parameterized by the homogeneous map
$\phi: \pp^1 \to \pp^2$ defined by 
$\phi_i(t, \theta) = t \binom{2}{i} \theta^i (1-\theta)^{2-i}$ for $i = 0, 1, 2$. The transposed Jacobian is
\[
J(\phi) =  \left[ \begin{array}{ccc}
    (1-\theta)^2 & 2 \theta (1-\theta) & \theta^2  \\
     -2 t (1-\theta) & 2t( 1 - 2\theta) & 2t\theta
\end{array} \right]. 
\]
Let $\mm_\phi$ denote the corresponding matroid which has ground set $\{0, 1, 2 \}$ corresponding to the columns of $J(\phi)$. The independent sets are sets $S \subseteq \{0, 1, 2 \}$ such that columns in $S$ are linearly independent over the fraction field $\cc(t, \theta)$. One can verify through direct computation that the independent sets are exactly $S \subseteq \{0, 1, 2 \}$ such that $\#S < 3$. 

On the other hand, the homogeneous vanishing ideal of $M$ is  $\ii(M) = 
\langle 4p_0p_2 - p_1^2 \rangle$. Its corresponding matroid, which we denote by $\mm_{\ii(M)}$, also has ground set $\{0, 1, 2\}$ and a set $S \subseteq \{0, 1, 2\}$ is an independent set in $\mm_{\ii(M)}$ if
$\ii(M) \cap \cc[S] = \langle 0 \rangle$ where $\cc[S] = \cc[p_i : i \in S]$. In this case it is straightforward to see that the independent sets are again the sets $S$ such that $\#S < 3$ and so $\mm_\phi =  \mm_{\ii(M)}$. 
\end{ex}

Note that, as we have done in Example \ref{ex:binomialrv}, we will usually
work with homogeneous vanishing ideals of algebraic statistical models.
This has the advantage of simplifying some computations, but does not affect the underlying theory.

%%%%%%%%%%%%%%%%%%%%%%%%%%%%%%%%%%%%%%%%%%%%%%
%%%%%%%%%%%%%%%%%%%%%%%%%%%%%%%%%%%%%%%%%%%%%%%
%%%%%%%%%%%%%%%%%%%%%%%%%%%%%%%%%%%%%%%%%%%%%%
%%%%%%%%%%%%%%%%%%%%%%%%%%%%%%%%%%%%%%%%%%%%%%%

\section{Certifying Generic Identifiability With Algebraic Matroids}
\label{sec:algID}
In this section we make a few basic observations that will lead to a 
new algorithm for certifying the generic identifiability of a 
family of models using their associated algebraic matroids. 
Our starting point for proving identifiability using algebraic methods
is Proposition \ref{prop:idealID}.  However, it often difficult to
find the polynomials required  by Proposition \ref{prop:idealID} to  certify identifiability.  
The following proposition is the driver of our algebraic matroid based procedure for 
verifying identifiability.

\begin{prop}
\label{prop:matID}
Let $M_1$ and $M_2$ be two irreducible algebraic models which sit inside the probability simplex $\Delta_r$. Without loss of generality assume $\dim(M_1) \geq \dim(M_2)$.
If there exists a subset $S$ of the coordinates such that
\begin{equation}
\label{eqn:matID}
S \in \mm(M_2)\setminus \mm(M_1) 
\end{equation}
then $\dim(M_1 \cap M_2) < \min(\dim(M_1), \dim(M_2))$.
\end{prop}

Note that we abuse notation and write $\mm(M)$ to denote the matroid $\mm(\overline{M})$.

\begin{proof}
Since $M_1$ and $M_2$ are irreducible their vanishing ideals $\ii(M_1)$ and $\ii(M_2)$ are prime and so define the same matroid as $M_1$ and $M_2$ respectively. First suppose that $\dim(M_1) > \dim(M_2)$. This dimension inequality implies that there is a polynomial $f_2 \in \ii(M_2) \setminus \ii(M_1)$. Then since $S \in \mm(M_2)\setminus \mm(M_1)$, it holds that
$\ii(M_1) \cap k[S] \neq \zeroideal$ but $\ii(M_2) \cap k[S] = \zeroideal$ which implies that there exists $f_1 \in \ii(M_1) \setminus \ii(M_2)$ and so the result follows by Proposition \ref{prop:idealID}. 

Now suppose that $\dim(M_1) = \dim(M_2)$. The existence of $S \in \mm(M_2)\setminus \mm(M_1)$ implies that $M_1 \neq M_2$. Two irreducible models of the same dimension must either be equal or have lower dimensional intersection so the inequality of $M_1$ and $M_2$ implies the result.
\end{proof}

In essence, Proposition  \ref{prop:matID} can certify the existence of the desired polynomials for 
applying Proposition \ref{prop:idealID}, without necessarily finding them, only proving they exist. 
Note that Proposition \ref{prop:matID} is weaker then 
Proposition \ref{prop:idealID}.
This because there can be  models with different ideals but that have  the same matroid. 
This is due to the fact that the matroid only keeps track of which sets of coordinates have polynomial relations in the ideals of the models but not the nature of the polynomial relations themselves.  
This is illustrated in Example \ref{ex:matroidIDFail}.

\begin{prop}
\label{prop:nmLocus}
\cite[Proposition 2.5]{compAlgMat14} Let $k$ be a field of characteristic zero and $V \subset k^n$ be a variety parameterized by $\phi$ with Jacobian $J(\phi)$ defined as in Equation (\ref{eqn:paramJac}). Then the matrix obtained by plugging in generic parameter values into $J(\phi)$ gives a linear matroid over $k$ which is the same as that defined by $J(\phi)$ with symbolic parameters over $k(\theta)$ and thus the same as $\mm(V)$.  
\end{prop}

We use $\mm(J(\phi),k)$ to denote the linear matroid we get by plugging in random parameter values for $\theta$ and $\mm(J(\phi),k(\theta))$ to denote the symbolic matroid. 
With these two propositions we are ready to define the main algorithm that we use to prove identifiability results.

\begin{alg}
\label{alg:certID}
Input: Two maps $\phi_1, \phi_2$ parameterizing models $M_1$ and $M_2$ in $k^n$ with $\dim(M_1) \geq \dim(M_2)$, a number of trials $t$. 
\newline
Output: A certificate $S$ satisfying Equation (\ref{eqn:matID}) in Proposition \ref{prop:matID}.
\newline
For $i = 1, 2, \ldots t$:
\begin{itemize}
    \item Randomly select $T \subseteq [n]$ such that $|T| \leq \dim(M_2)$.
    
    \item If $T \in \mm(J(\phi_2), k)\setminus \mm(J(\phi_1),k)$:
    \begin{itemize}
        \item If $T \in \mm(J(\phi_2), k(\theta))\setminus \mm(J(\phi_1),k(\theta))$:
        \begin{itemize}
            \item Then, $S = T$.
        \end{itemize}
        
    \end{itemize}
\end{itemize}
Output: S or report that no certificate was found.
\end{alg}
In summary, the algorithm works by randomly plugging in a numerical value for $\theta$
and testing random subsets until it finds an example of a set $S$ where the submatrices
of the Jacobians have different rank.  Random rational numbers are used so that
the rank computations are exactly calculated symbolically (rather than using a numerical
rank test with floating point numbers).  Once a candidate set is found, then
an exact symbolic computation over $k(\theta)$ is performed to verify the
result exactly.

In cases where it is too time consuming to compute over $k(\theta)$ we can use the Schwartz-Zippel Lemma from polynomial identity testing to produce a certificate that satisfies equation (\ref{eqn:matID}) with probability $1 - \epsilon$.

\begin{lemma}(Schwartz-Zippel) Let $f \in k[x_1, \ldots x_n]$ be a non-zero polynomial of total degree $\alpha$. Let $E$ be a finite subset of $k$ and $r_1, \ldots r_n$ be selected at random independently and uniformly from $E$. Then
\[
P(f(r_1, \ldots, r_n) = 0) \leq \frac{\alpha}{|E|}.
\]
\end{lemma}

Determining if $S \in \mm(J(\phi_2), k(\theta))\setminus \mm(J(\phi_1),k(\theta))$ can be done by evaluating minors of the submatrices of the Jacobian matrices corresponding to $S$. Since minors are polynomials in the entries of the matrices, we can use the this lemma to bound the probability that $S \in \mm(J(\phi_2), k(\theta))\setminus \mm(J(\phi_1),k(\theta))$ without ever computing over $k(\theta)$. 

\begin{cor}
\label{cor:matroidSZ}
Let $S \in \mm(J(\phi), k(\theta))$ and let $E \subseteq k$ be a finite set such that $|E| > \alpha$ where $\alpha$ is the degree of an $|S| \times |S|$ minor of $J(\phi)_S$ that is not identically zero. Let $r_1, \ldots r_d$ be selected independently and uniformly at random from $E$ and let $J = J(\phi)|_r$ be the specialization of $J(\phi)$ at $r_1 \ldots r_d$. 
Then
\[
P(S \notin \mm(J)) \leq \frac{\alpha}{|E|}.
\]
\end{cor}
\begin{proof}
First note that the $|S| \times |S|$ minors of the matrix $J(\phi)_S$ are polynomials in $\theta$ which we denote by $f_i(\theta) \in k[\theta]$ for $1 \leq i \leq \binom{d}{|S|}$. Since $S \in \mm(J(\phi), k(\theta)) $ there exists at least one $f_j$ that is not identically zero. On the other hand, $S \notin \mm(J)$ if and only if $f_i(r) = 0$ for all $1 \leq i \leq \binom{d}{|S|}$ so
\[
P( S \notin \mm(J)) = P \left(f_i(r) = 0, 1
\leq i \leq \binom{d}{|S|} \right) \leq P(f_j(r) = 0)
\]
Letting $\deg(f_j) = \alpha$ and applying the Schwartz-Zippel Lemma gives
\[
P(f_j(r) = 0) \leq \frac{\alpha}{|E|}
\]
which gives us the desired result. 
\end{proof}

Corollary \ref{cor:matroidSZ} implies that the probability that $S$ is not independent in the original matroid $\mm(J(\phi), k(\theta))$ is $1 - \frac{\alpha}{|E|}$. We can make this probability even larger using amplification by independent trials. This naturally leads to the following algorithm which does not require any symbolic computation.

\begin{alg}
\label{alg:certID_SZ}
Input: Two maps $\phi_1, \phi_2$ parameterizing models $M_1$ and $M_2$ in $k^n$ with $\dim(M_1) \geq \dim(M_2)$, a number of trials $t$, a tolerance $\epsilon$. 
\newline
Output: A certificate $S$ satisfying Equation (\ref{eqn:matID}) in Proposition \ref{prop:matID} with probability at least $1 - \epsilon$. 
\newline
Initialize: Choose a finite subset $|E| \subseteq k$ such that $|E| > \alpha$ where $\alpha$ is the maximum degree of any $dim(M_2) \times dim(M_2)$ minor of $J(\phi_1)$.
\newline
For $i = 1, 2, \ldots t$:
\begin{enumerate}

    \item Randomly select $T \subseteq [n]$ such that $|T| \leq \dim(M_2)$.
    
    \item Sample points $r_1, \ldots r_d$ independently and uniformly at random from $E$.
    
    \item If $T \in \mm(J(\phi_2)|_r)\setminus \mm(J(\phi_1)|_r)$:
    \begin{enumerate}
        \item Choose $l$ such that $\left(\frac{\alpha}{|E|} \right)^l \leq \epsilon$.
        \item For $j = 1, 2, \ldots l$:
        \begin{itemize}
            \item Sample points $r_1', \ldots r_d'$ independently and uniformly at random from $E$.

            \item If $T \in \mm(J(\phi_1)|_r)$:
            \begin{itemize}
                \item Break and return to (1).
            \end{itemize}
        \end{itemize}
        
        \item Then, $S = T$.
        
    \end{enumerate}
\end{enumerate}
Output: S or announce that no certificate was found.
\end{alg}

Both Algorithm \ref{alg:certID} and Algorithm \ref{alg:certID_SZ} c
an be modified in the case that $\dim(M_1) = \dim(M_2)$. 
In that case, we can also accept $T$ as a certificate if it is an 
independent set for $M_1$ but not $M_2$ in both the numerical step and the symbolic step. 
This is because we just need to certify that the models are not equal in the case 
where they are the same dimension. This is a very general version of the algorithm 
and it can be fine tuned in many ways depending on the specifics of the models. 
One such modification in the case that the models are the same dimension would be 
to only check for sets $T$ such that $|T| = \dim(M_2)$. This is equivalent to searching 
for a basis for the matroid of one model that is not a basis for the other. 
If the matroids are not the same, then such a basis must exist since a matroid 
is uniquely determined by its bases.   However, from a practical standpoint, it is
faster to perform the symbolic rank calculations on smaller matrices, so hunting for
small sized sets that verify that the matroids are different can speed up computations.  

Using Algorithm \ref{alg:certID} and Proposition \ref{prop:matID} to certify 
identifiability has several advantages over approaches that rely on 
Proposition \ref{prop:idealID}. Algorithm \ref{alg:certID}  does not require an 
implicit description of the models $M_1$ and $M_2$ so time-consuming elimination 
computations are avoided. Symbolic computation is also only done to verify that a test set 
$T$ is in fact a certificate but not to find the candidate set. Proposition \ref{prop:nmLocus} 
guarantees that if there is a certificate $S$ that can be found symbolically, 
then it can be found numerically with probability 1 so we can minimize the 
amount of symbolic computation necessary. Lastly, we are frequently able to 
find certificates by just randomly searching for them which avoids the combinatorial complexity of computing the whole matroid. The downside of this is that the failure of the algorithm does not imply that the matroids are the same or that a discrete parameter is not identifiable. This type of failure is illustrated by Example \ref{ex:matroidIDFail}.

%%%%%%%%%%%%%%%%%%%%%%%%%%%%%%%%%%%%%%%%%%%%%%%%%%%%%%%%%%%%%%%%%%%%
%%%%%%%%%%%%%%%%%%%%%%%%%%%%%%%%%%%%%%%%%%%%%%%%%%%%%%%%%%%%%%%%%%%%
%%%%%%%%%%%%%%%%%%%%%%%%%%%%%%%%%%%%%%%%%%%%%%%%%%%%%%%%%%%%%%%%%%%%
%%%%%%%%%%%%%%%%%%%%%%%%%%%%%%%%%%%%%%%%%%%%%%%%%%%%%%%%%%%%%%%%%%%%

\section{Identifiability Of 2-tree mixtures for Generic Group-Based Models}
\label{sec:phylo}
In this section we demonstrate how Algorithm \ref{alg:certID} 
can be used to certify the identifiability of the tree parameters 
in group-based phylogenetic models. In Section \ref{sec:phyloNetID} we apply the
method to the identifiability of phylogenetic network models.
We begin with some basic background on phylogenetic models on trees. 

Many of our proofs in the following sections use supplementary files. We will reference relevant supplementary files or methods as needed. All of these files are located at the website:
\begin{center}
    \textbf{https://github.com/bkholler/MatroidIdentifiability}
\end{center}

\subsection{Preliminaries on Phylogenetic Models}
A $\kappa$-state phylogenetic Markov model on a $n$-leaf, leaf-labelled rooted binary tree $T$ gives us a joint distribution on the states of the leaves of $T$. This joint distribution is determined by associating a $\kappa$-state random variable $X_v$ to each internal vertex $v$ of $T$ and a $\kappa \times \kappa$ transition matrix $M^e$ to each directed edge $e = (u,v)$ of $T$ such that $M_{i,j}^e = P(X_v = j | X_u = i)$. A root distribution $\pi$ for the root $\rho$ of $T$ is also needed. The transition matrices $\{M^e\}_{e\in E(T)}$ and the root distribution $\pi$ are called the continuous parameters of the model. 

We let $[\kappa]$ be the state space of these random variables and $Int(T)$ be the set of internal vertices of $T$. Also let $X_i$ be the random variable associated to the leaf labelled $i$ for $i \in [n]$. Then the probability of observing a configuration 
$(x_1, \ldots x_n) \in [\kappa]^n$ of states at the leaves is 
\[
P(X_1 = x_1, \ldots, X_n = x_n) ~= 
\sum_{j \in [\kappa]^{Int(T)}}\pi_{j_\rho}\prod_{(u,v) \in E(T)}M_{j_u, j_v}^{(u,v)}. 
\]

\begin{ex}
Let $T$ be the three leaf tree pictured in Figure \ref{fig:threeLeafTree}. The random variables $Y_1$ and $Y_2$, which correspond to internal nodes, are hidden random variables of the model whereas the random variables $X_1, X_2, X_3$, which correspond to leaves, are observed. 

We let $M^i$ be transition matrices associated to each edge as pictured in Figure \ref{fig:threeLeafTree}. The transition matrix $M^i$ gives the probability of the random variables changing states along the corresponding edge. For instance, if we let $i, j \in [\kappa]$, then $P(X_1 = j | Y_1 = i) = M_{i,j}^1$. Lastly we choose a root distribution $\pi$ to be the distribution of the random variable $Y_1$. Then the probability of observing $(x_1, x_2, x_3) \in [\kappa]^3$ at the leaves is
\[
P(X_1 = x_1, X_2 = x_2, X_3 = x_3) ~= 
\sum_{(y_1, y_2) \in [\kappa]^2} \pi_{y_1} M_{y_1, y_2}^0 M_{y_1, x_1}^1 M_{y_2, x_2}^2 M_{y_2, x_3}^3.
\]
The first coordinate of $(y_1, y_2) \in [\kappa]^2$ corresponds to the root which has associated random variable $Y_1$. The second coordinate corresponds to the other internal vertex which has associated random variable $Y_2$.
\end{ex}

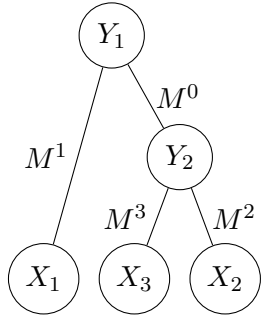
\begin{figure}
    \centering
    \begin{forest}
    for tree={circle, draw, l sep=20pt}
    [$Y_1$ 
        [$X_1$ , tier = leaf, edge label = {node[midway, left] {$M^1$}}
        ] 
        [$Y_2$, edge label = {node[midway, right] {$M^0$}}
            [$X_3$ , tier = leaf,  edge label = {node[midway, left] {$M^3$}}] 
            [$X_2$ , tier = leaf, edge label = {node[midway, right] {$M^2$}}]
        ]
    ]
    \end{forest}
   
    \caption{A three leaf tree with a random variable associated to each node of the tree. The matrices $M^i$ are the transition matrices encoding the probabilities of the random variables changing states.}
    \label{fig:threeLeafTree}
\end{figure}

We can see that the joint distribution of $(X_1, \ldots X_n)$ is given by polynomials in the entries of $\pi$ and the $M^e$. In other words, the model can be thought of as a polynomial map
\[
\psi_T : \Theta_T \to \Delta_{\kappa^n-1}
\]
where $\Theta_T$ is the \emph{stochastic parameter space} of the model and $\Delta_{\kappa^n-1}$ is the probability simplex. We can also consider the variety $V_T$ we get by taking the Zariski closure of the image of $\psi_T$. Polynomials in the vanishing ideal $\ii(V_T)$ are called \emph{phylogenetic invariants} and were first studied in
\cite{cavenderInv87, lakeInv87}. For more information on these models we refer the reader to \cite{steel2016phylogeny}. 

With such a model, the 2-tree mixture model for trees $T_1$ and $T_2$ leaf-labelled by $[n]$ is obtained by taking the image of the map
\[
\psi_{T_1,T_2}: \Theta_{T_1} \times \Theta_{T_2} \times [0,1] \to \Delta_{\kappa^n-1}
\]
defined by
\[
\psi_{T_1,T_2}(\theta_1,\theta_2,\lambda) = \lambda \psi_{T_1}(\theta_1) + (1-\lambda)\psi_{T_2}(\theta_2). 
\]
The 2-tree mixture model is the image of the map $\psi_{T_1,T_2}$ but the main object of interest is the variety naturally obtained by taking the Zariski closure of the image. Denote this variety by $V_{T_1} \ast V_{T_2}$ which is the \emph{join variety} of the varieties $V_{T_1}$ and $V_{T_2}$. For additional information of join varieties we refer the reader to \cite{harrisAlgGeo}.

\subsection{Group-Based Phylogenetic Models in Fourier Coordinates}
\label{sec:gpBasedModels}
Group-based models are a family of phylogenetic Markov models where the random variables associated to each vertex take values in a finite abelian group. This allows for a linear change of coordinates in which the models are given by monomial maps. 

\begin{defn}
Let $G$ be a finite abelian group of order $\kappa$ and $T$ a rooted binary tree. Then a group-based model on $T$ is a phylogenetic Markov model on $T$ such that for each transition matrix $M^e$, there exists a function $f_e : G \to \rr$ such that $M_{g,h}^e = f(g-h)$. 
\end{defn}

As mentioned above, we think of the random variables $X_v$ as taking values in the group $G$, and the transition matrices as being indexed by the elements of the group.
We will focus on the Cavendar-Farris-Neyman (CFN), Jukes-Cantor (JC), Kimura 2-Parameter (K2P), and Kimura 3-Parameter (K3P) models. The CFN model is associated to the group $\zz_2$ while the other three models are associated to the group $\zz_2 \times \zz_2$. The form of the transition matrices for the CFN and K3P models are pictured in Figure \ref{fig:gpBasedMats}. 
\begin{figure}
  \begin{minipage}{.25\linewidth}
    \centering
    \[\left[\begin{array}{cc}
      \alpha & \beta \\
      \beta & \alpha
    \end{array}\right]\]
   CFN
  \end{minipage}
  \begin{minipage}{.25\linewidth}
    \centering
    \[\left[\begin{array}{cccc}
      \alpha & \beta   & \gamma & \delta \\
      \beta  & \alpha  & \delta & \gamma \\
      \gamma & \delta  & \alpha & \beta \\
      \delta & \gamma  & \beta  & \alpha \\
    \end{array}\right]\]
    K3P
  \end{minipage}
  \caption{Transition matrices in the CFN and K3P models have the above forms}
  \label{fig:gpBasedMats}
\end{figure}

Group-based models allow for a linear change of coordinates that makes $\psi_T$ a monomial map, thus the variety $V_T$ is a \emph{toric variety} \cite{toricPhyloInv05}. This change of coordinates is called the discrete Fourier transform and was first applied to phylogenetic models in \cite{evansPhyloInv93, hendyPhyloInv96}. The new image coordinates, commonly called the Fourier coordinates, are denoted with $q_{g_1,\ldots, g_n}$ for $g_1,\ldots,g_n \in G$. This map is defined even more simply in the case that $G$ is $\zz_2$ or $\zz_2 \times \zz_2$ which we will restrict to. In this case, the map can be described in terms of the \emph{splits} of the tree which we briefly describe first.

A split of $[n]$ is a set partition $A|B$ of the set $[n]$. A split $A|B$ is valid for an unrooted binary tree $T$ leaf-labelled by $[n]$ if it can be obtained as the leaf sets of the two connected components of $T \setminus e$ for some edge $e$ of $T$. Furthermore, every such tree $T$ is uniquely determined by its set of splits which we denote by $\Sigma(T)$ \cite[Theorem 15.1.6]{algStat18}.

Now for each split $A|B \in \Sigma(T)$ and each group element $g \in G$ we have a parameter $a_g^{A|B}$. The parameterization of the model $\psi_T$ in the Fourier coordinates is given by
\begin{equation}
q_{g_1,\ldots g_n} = 
\begin{cases}
\prod_{A|B \in \Sigma(T)} a_{\sum_{i \in A}g_i}^{A|B} & \mbox{ if } \sum_{i \in [n]}g_i = 0 \\ 
0 & \mbox{ otherwise}
\end{cases}
\end{equation}
In the JC and the K2P models, further conditions are imposed on the parameters $a_g^{A|B}$ but in the \emph{generic group based models}, which are the CFN and K3P models, there are no other restrictions on the parameters. 

\begin{ex}
Let $T_1$ be the tree pictured in Figure \ref{fig:badTreeCase}. The nontrivial splits of $T_1$ are 
\newline
$\{12|3456, 123|456, 1234|56 \}$. Since each split is a set partition of $[6]$ into two parts, we can just use one of the parts of the set partition to denote the parameter corresponding to that split. So the parameterization $\psi_{T_1}$ in the Fourier coordinates will be
\[
q_{g_1, \ldots g_6} = \begin{cases}
a_{g_1}^1 a_{g_2}^2 a_{g_3}^3 a_{g_4}^4 a_{g_5}^1 a_{g_6}^6 a_{g_1 + g_2}^{12} a_{g_1 + g_2 + g_3}^{123} a_{g_5 + g_6}^{56} & \mbox{ if } \sum_{i \in [6]}g_i = 0 \\ 
0 & \mbox{ otherwise}
\end{cases}
\]
\end{ex}

The linearity of the Fourier transform allows us to also apply this change of coordinates to 2-tree mixture models as well \cite{mixJCID10} which makes the map $\psi_{T_1,T_2}$ into a binomial map. So we can view these mixture models as a family of algebraic models indexed by a discrete parameter which is 2-multisets of $[n]$-leaf trees. This leads to the following definition of generic identifiability of the tree parameters for 2-tree mixture models which is essentially a specialized version of Definition \ref{def:genID}.

\begin{defn}
\label{def:treeId}
The tree parameters of a 2-tree mixture model are \emph{generically identifiable} if for every pair of distinct multisets of $n$-leaf trees $\{T_1,T_2\}$ and $\{S_1,S_2\}$,
\[
\dim((V_{T_1} \ast V_{T_2}) \cap (V_{S_1} \ast V_{S_2})) <
\min(\dim(V_{T_1} \ast V_{T_2}),~\dim(V_{S_1} \ast V_{S_2})).
\]
\end{defn}

\subsection{Identifiability of Tree Parameters in 2-tree CFN and K3P Mixtures}
\label{sec:idMixCFN}
In this section we will discuss how Algorithm \ref{alg:certID} can be specialized for separating 2-tree CFN mixtures of 6-leaf trees and show how it can be used to prove generic identifiability of the tree parameters CFN model when combined with the following theorem of Matsen, Mossel, and Steel \cite{mixedUpTrees08}. All of the computations involved are available in the supplementary materials. 

\begin{thm}(Six-To-Infinity Theorem)
\cite[Theorem 23]{mixedUpTrees08}
\label{thm:sixToInfinity}
Suppose that the tree parameters $T_1,T_2$ are identifiable for a 2-tree mixture model for trees with six leaves. Then the tree parameters are identifiable for trees with $n$ leaves for all $n \geq 6$. 
\end{thm}

\begin{thm}
\label{thm:CFN_ID}
The tree parameters of the 2-tree CFN mixture model are generically identifiable for trees with at least 6 leaves.
\end{thm}
\begin{proof}
If we can show that the tree parameters are identifiable for 2-tree CFN mixtures of six leaf trees then we are done by the Six-To-Infinity Theorem. Our proof of this is computational and simply an application of Algorithm \ref{alg:certID} with some simplifications.

First, we note that instead of comparing every possible pair of 2-multisets of 
six leaf trees of which there $15,481,830$ it is enough to check up to the symmetry induced
by the permutation action of $S_6$ on leaf labels.
Consideration of symmetry reduces the problem to checking only $22,773$ distinct cases.

Next we note that 2-tree CFN mixtures of six leaf trees have the 
\emph{expected dimension} \cite[Proposition 5.5]{dimGpBasedMix19}. This means that for every pair of six leaf trees $\{T_1, T_2 \}$, the variety $V_{T_1} \ast V_{T_2}$ satisfies
\[
\dim(V_{T_1} \ast V_{T_2}) = \dim(V_{T_1}) + \dim(V_{T_2}) + 1 = 19
\]
which implies that join varieties of this form have the same dimension regardless 
of the tree parameters. This means we can use the specialized version of 
the algorithm for models of the same dimension. Furthermore, as a result of the Fourier transform, 
the parameters that correspond to the identity element in $\zz_2$ are actually identically 1. 
By removing them we are able to greatly reduce the number of variables 
which significantly speeds up the symbolic computation step required for verification of certificates. 
Our algorithm is able to produce a certificate for all but one of the 22,773 cases. 
These certificates are stored in the file \emph{certsCFN} and code to verify that they are certificates can be found in the Mathematica file \emph{CFN\_6Leaf\_Mixtures.nb}. The certificates were originally found using the function \emph{matroidSeparate} in the Mathematica package \emph{PhylogeneticMatroids.m} which is our implementation of Algorithm \ref{alg:certID}. 

The case that the algorithm fails to find a certificate has tree parameters $\{T_1,T_2\}$ and $\{S_1,S_2\}$ of the following form up to symmetry
\begin{equation}
\label{eqn:badTreeCase}
\begin{aligned}
T_1 &= \{12|3456, 123|456, 1234|56 \} \\ 
T_2 &= \{23|1456, 123|456, 1236|45\} \\
S_1 &= \{12|3456, 123|456, 1236|45 \} \\ 
S_2 &= \{23|1456, 123|456, 1234|56\}. 
\end{aligned}
\end{equation}

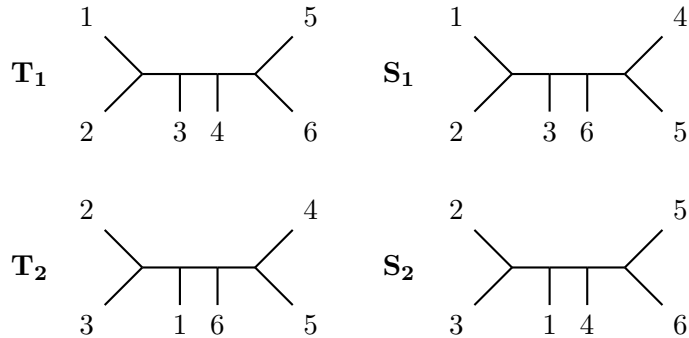
\begin{figure}
    \centering
    \begin{subfigure}[b]{0.3\linewidth}
        \centering
        \begin{tikzpicture}[scale = .5, thick]
        %\draw [help lines] (0,0) grid (5,5);
        \draw (0,4)--(1,3);
        \draw (0,2)--(1,3);
        \draw (1,3)--(2,3);
        \draw (2,3)--(3,3);
        \draw (2,3)--(2,2);
        \draw (3,3)--(4,3);
        \draw (3,3)--(3,2);
        \draw (4,3)--(5,4);
        \draw (4,3)--(5,2);
        
        \draw (0,4) node[above left]{$1$};
        \draw (0,2) node[below left]{$2$};
        \draw (2,2) node[below]{$3$};
        \draw (3,2) node[below]{$4$};
        \draw (5,4) node[above right]{$5$};
        \draw (5,2) node[below right]{$6$};
        
        \draw (-2,3) node{$\mathbf{T_1}$};
        \end{tikzpicture}
    \end{subfigure}
    \begin{subfigure}[b]{0.3\linewidth}
        \centering
        \begin{tikzpicture}[scale = .5, thick]
        %\draw [help lines] (0,0) grid (5,5);
        \draw (0,4)--(1,3);
        \draw (0,2)--(1,3);
        \draw (1,3)--(2,3);
        \draw (2,3)--(3,3);
        \draw (2,3)--(2,2);
        \draw (3,3)--(4,3);
        \draw (3,3)--(3,2);
        \draw (4,3)--(5,4);
        \draw (4,3)--(5,2);
        
        \draw (0,4) node[above left]{$1$};
        \draw (0,2) node[below left]{$2$};
        \draw (2,2) node[below]{$3$};
        \draw (3,2) node[below]{$6$};
        \draw (5,4) node[above right]{$4$};
        \draw (5,2) node[below right]{$5$};
        \draw (-2,3) node{$\mathbf{S_1}$};
        \end{tikzpicture}
    \end{subfigure}
    
    \vspace{5mm}
    
    \begin{subfigure}[b]{0.3\linewidth}
        \centering
        \begin{tikzpicture}[scale = .5, thick]
        %\draw [help lines] (0,0) grid (5,5);
        \draw (0,4)--(1,3);
        \draw (0,2)--(1,3);
        \draw (1,3)--(2,3);
        \draw (2,3)--(3,3);
        \draw (2,3)--(2,2);
        \draw (3,3)--(4,3);
        \draw (3,3)--(3,2);
        \draw (4,3)--(5,4);
        \draw (4,3)--(5,2);
        
        \draw (0,4) node[above left]{$2$};
        \draw (0,2) node[below left]{$3$};
        \draw (2,2) node[below]{$1$};
        \draw (3,2) node[below]{$6$};
        \draw (5,4) node[above right]{$4$};
        \draw (5,2) node[below right]{$5$};
        
        \draw (-2,3) node{$\mathbf{T_2}$};
        \end{tikzpicture}
    \end{subfigure}
    \begin{subfigure}[b]{0.3\linewidth}
        \centering
        \begin{tikzpicture}[scale = .5, thick]
        %\draw [help lines] (0,0) grid (5,5);
        \draw (0,4)--(1,3);
        \draw (0,2)--(1,3);
        \draw (1,3)--(2,3);
        \draw (2,3)--(3,3);
        \draw (2,3)--(2,2);
        \draw (3,3)--(4,3);
        \draw (3,3)--(3,2);
        \draw (4,3)--(5,4);
        \draw (4,3)--(5,2);
        
        \draw (0,4) node[above left]{$2$};
        \draw (0,2) node[below left]{$3$};
        \draw (2,2) node[below]{$1$};
        \draw (3,2) node[below]{$4$};
        \draw (5,4) node[above right]{$5$};
        \draw (5,2) node[below right]{$6$};
        \draw (-2,3) node{$\mathbf{S_2}$};
        \end{tikzpicture}
    \end{subfigure}

    \caption{The two pairs of trees described by Equation (\ref{eqn:badTreeCase}) which have the same sets of splits when combined.}
    \label{fig:badTreeCase}
\end{figure}

In this case, we were able to find invariants that separate the join varieties by computing a degree-bounded Gr\"{o}bner basis for the varieties $V_{T_1} \ast V_{T_2}$ and
$V_{S_1} \ast V_{S_2}$ up to degree 4. These computations can be found in \emph{CFN\_last\_pair.m2}. This separates all pairs up to symmetry and so the tree parameters are identifiable for six leaf trees. 
\end{proof}

A natural question to ask is why the more typical Gr\"{o}bner basis algorithm that was employed to deal with the last case could not simply be used to deal with every case. This is because even the degree bounded Gr\"{o}bner basis calculation can take a significant amount of time compared to our algorithm. For instance if we take
\begin{equation*}
\begin{aligned}
T_1 &= \{12|3456, 125|346, 1256|34 \} \\ 
T_2 &= \{13|2456, 134|256, 1346|25\} \\
S_1 &= \{12|3456, 126|345, 1246|35 \} \\ 
S_2 &= \{15|2346, 156|234, 1356|24\}. 
\end{aligned}
\end{equation*}
then computing a Gr\"{o}bner basis up to degree four took slightly over eight minutes whereas our algorithm took slightly under four minutes in this case. This computational difference is quite significant given the large number of cases that need to be dealt with.

For the main computation we ran Algorithm \ref{alg:certID} on each case in batches of about 1000 cases over a month. We do not have a precise time estimate for how long this computation took but in the 22,772 cases where the algorithm worked, it seems to find potential certificates quickly and most of the computation time came from computing matrix rank over the fraction field $k(\theta)$. On the other hand, using Algorithm 3.6 with tolerance $\epsilon = 10^{-10}$, we can find a list of certificates in slightly over 19 minutes running the algorithm in parallel on a laptop with four processors.

The identifiability of the tree parameters for 2-tree K3P mixtures actually follows almost immediately from the CFN case but we are also able to use our method along with some results from \cite{mixJCID10} to get identifiability results for smaller trees in the K3P case. 

\begin{thm}
\label{thm:K3P_ID}
The tree parameters of the 2-tree K3P mixture model are generically identifiable for trees with at least four leaves.
\end{thm}
\begin{proof}
The generic identifiability of the tree parameters of 2-tree K3P mixtures for trees with at least six leaves follows immediately from Theorem \ref{thm:CFN_ID}. This is because the CFN model can be obtained from the K3P model via a coordinate projection. More explicitly, let $\{T_1, T_2\}$ and $\{S_1, S_2\}$ be two distinct multisets of six leaf trees
and 
suppose $V_{T_1}^{K3P} \ast V_{T_2}^{K3P}$ and $V_{S_1}^{K3P} \ast V_{S_2}^{K3P}$ are the join varieties associated to the K3P model. Theorem \ref{thm:CFN_ID} guarantees that the same varieties associated the CFN model satisfy
\[
V_{T_1}^{CFN} \ast V_{T_2}^{CFN} \neq V_{S_1}^{CFN} \ast V_{S_2}^{CFN}.
\]
Let $G$ be a subgroup of $\zz_2 \times \zz_2$ isomorphic to $\zz_2$, and $\pi : \cc^{4^n} \to \cc^{2^n}$ be the linear map obtained by projecting onto the coordinates of $\cc^{4^n}$ indexed only by the elements of $G$. Then for any tree $T$, $\pi(V_T^{K3P}) = V_T^{CFN}$.
For example, let $G = \langle (1,0) \rangle \subseteq \zz_2 \times \zz_2$ and let $\pi(V_T^{K3P})$ be the projection onto these coordinates. Then $\pi(V_T^{K3P}) \subseteq \cc^{2^n}$ is parameterized by the map
\[
q_{g_1, \ldots g_n} =
\begin{cases}
\prod_{A|B \in \Sigma(T)} a_{\sum_{i \in A}g_i}^{A|B} & \mbox{, if } \sum_{i \in [n]}g_i = 0 \\ 
0 & \mbox{, otherwise}
\end{cases}
\]
since $G \cong \zz_2$, we can simply replace every occurence of $(1,0) \in \zz_2 \times \zz_2$ in this map with $1 \in \zz_2$ without changing the map at all. This replacement gives the parameterization of the variety $V_T^{CFN}$ and so we see that the two varieties are parameterized by the same map. Since linear maps commute with taking joins of varieties, it holds that 
\[
\pi \left( V_{T_1}^{K3P} \ast V_{T_2}^{K3P} \right) = V_{T_1}^{CFN} \ast V_{T_2}^{CFN}.
\]
This together with the inequality of the 2-tree CFN join varieties implies the inequality of the 2-tree K3P join varieties and so the generic identifiability of the tree parameters for trees with at least six leaves follows. 

For trees with four leaves we once again apply Algorithm \ref{alg:certID} to all distinct 2-multisets of four leaf trees up to symmetry. In all four cases the algorithm quickly finds a certificate numerically but in the last case it seems to only find certificates that are very large. When the potential certificate sets are large, the verification step can still be time consuming so in this case we instead constructed a smaller certificate that we verified symbolically. The five leaf case then follows from Proposition 7 of \cite{mixJCID10} which guarantees that if $\{T_1, T_2 \}$ and $\{S_1, S_2 \}$ are distinct multisets of five leaf trees, then there exists a 4-subset $K \subseteq [5]$ of the leaves such that restricting each tree to $K$ gives two distinct multisets of four leaf trees or in symbols $\{T_1|_K, T_2|_K \} \neq \{S_1|_K, S_2|_K \}$. The result then follows from Lemma 3 of \cite{mixJCID10} which shows that if $V_{T_1|_K} \ast V_{T_2|_K} \not\subseteq V_{S_1|_K} \ast V_{S_2|_K}$ then
$V_{T_1} \ast V_{T_2} \not\subseteq V_{S_1} \ast V_{S_2}$. 
\end{proof}

In the proof of Theorem \ref{thm:CFN_ID}, there was a single pair of trees up to symmetry 
that our matroid-based algorithm failed to find a certificate for. 
We also attempted to run the algorithm for the same pair of trees under the 
K3P model and also did not find a certificate but in both cases we know the corresponding ideals of phylogenetic invariants are not equal. As mentioned in section \ref{sec:algID}, 
it is possible for two different ideals to define the same matroid.  We conjecture
this to be the case in this instance.

\begin{conj}
Let $\{T_1, T_2\}$ and $\{S_1, S_2\}$ be the pairs of trees defined in Equation (\ref{eqn:badTreeCase}) and let $V_{T_1} \ast V_{T_2}$ and  $V_{S_1} \ast V_{S_2}$ be the associated CFN join varieties. Then 
\[
\mm(V_{T_1} \ast V_{T_2}) = \mm(V_{S_1} \ast V_{S_2}).
\]
\end{conj}

%%%%%%%%%%%%%%%%%%%%%%%%%%%%%%%%%%%%%%%%%%%%%%%%%%%%%%%%%
%%%%%%%%%%%%%%%%%%%%%%%%%%%%%%%%%%%%%%%%%%%%%%%%%%%%%%%%%
%%%%%%%%%%%%%%%%%%%%%%%%%%%%%%%%%%%%%%%%%%%%%%%%%%%%%%%%%
%%%%%%%%%%%%%%%%%%%%%%%%%%%%%%%%%%%%%%%%%%%%%%%%%%%%%%%%%

\section{Identifiability for Phylogenetic Networks}
\label{sec:phyloNetID}
Recently phylogenetic network models have emerged as a tool to account for 
events in the evolutionary history of organisms that trees cannot represent.
Non-treelike evolutionary processes include 
horizontal gene transfer and hybridization \cite{geneTreeMadison97,horizontalSyvanen94}. 
Similar to the case of trees, an important question to address is the identifiability of the 
network parameter in network-based phylogenetic models. 
Gross and Long showed in \cite{phyloNetId18} that the network parameter is
identifiable in \emph{large-cycle} JC network models by explicitly computing the associated ideals. 
We will show how Algorithm \ref{alg:certID} can be used to extend their results to large-cycle 
K2P and K3P network models.

%%%
%%%
%%%

\subsection{Preliminaries on Phylogenetic Networks}
In this section we provide some background on phylogenetic networks.  
Specifically we describe the basic structure of a phylogenetic network 
and the parameterization associated to a network. 
The following notation and terminology is adapted from \cite{phyloNetId18}. 

\begin{defn}
A phylogenetic network $N$ on leaf set $[n]$ is a rooted acyclic digraph with no edges in parallel and satisfying the following properties:
\begin{enumerate}
    \item the root has out-degree two;
    \item a vertex with out-degree zero has in-degree one, and the set of vertices with out-degree zero is $[n]$;
    \item all other vertices have either in-degree one and out-degree two, or in-degree two and out-degree one. 
\end{enumerate}
\end{defn}

Vertices with in-degree one and out-degree two are called \emph{tree vertices} and vertices with in-degree two and out-degree one are called \emph{reticulation vertices}. Edges directed into a reticulation vertex are called \emph{reticulation edges} and all other edges are called \emph{tree edges}. In this paper we will be focusing on group-based models on phylogenetic networks which are all \emph{time-reversible}. This means that it is impossible to identify the location of the root under this model so we are only interested in the underlying \emph{semi-directed} network structure of a phylogenetic network. The underlying semi-directed network is obtained from a phylogenetic network by suppressing the root and undirecting all tree edges in the network. 

As the number of reticulation vertices in the network increases, the parameterization of the model becomes increasingly complicated, thus it is natural to first focus on networks with only one reticulation vertex 
\cite{phyloNetId18}.

\begin{defn}
A cycle-network is a semi-directed network with one reticulation vertex. A $k$-cycle network is a cycle-network with cycle size $k$. Every $k$-cycle network can be built by attaching a binary tree with at least one leaf to every vertex of a $k$-cycle and specifying a single vertex of the cycle as the reticulation vertex. 
\end{defn}

\begin{ex}
The networks pictured in Figure \ref{fig:netL4C4} are both examples of  4-cycle networks. The reticulation vertex for the first network is the vertex where the leaf labelled 1 is attached. The two dotted edges adjacent to it are the \emph{reticulation edges}. Removing either of these edges gives an unrooted four leaf tree. 
\end{ex}

\begin{figure}
    \begin{subfigure}[b]{0.3\linewidth}
        \centering
        \begin{tikzpicture}[scale = .5, thick]
        %\draw [help lines] (0,0) grid (6,6);
        \draw [dashed] (2,2)--(4,2);
        \draw (4,2)--(4,4);
        \draw (4,4)--(2,4);
        \draw [dashed] (2,4)--(2,2);
        
        \draw (2,2)--(1,1);
        \draw (4,2)--(5,1);
        \draw (4,4)--(5,5);
        \draw (2,4)--(1,5);
        
        \draw (1,1) node[below]{$1$};
        \draw (5,1) node[below]{$2$};
        \draw (5,5) node[above]{$3$};
        \draw (1,5) node[above]{$4$};
        \end{tikzpicture}
    \end{subfigure}
    \begin{subfigure}[b]{0.3\linewidth}
        \centering
        \begin{tikzpicture}[scale = .5, thick]
        %\draw [help lines] (0,0) grid (6,6);
        \draw [dashed] (2,2)--(4,2);
        \draw (4,2)--(4,4);
        \draw (4,4)--(2,4);
        \draw [dashed] (2,4)--(2,2);
        
        \draw (2,2)--(1,1);
        \draw (4,2)--(5,1);
        \draw (4,4)--(5,5);
        \draw (2,4)--(1,5);
        
        \draw (1,1) node[below]{$1$};
        \draw (5,1) node[below]{$2$};
        \draw (5,5) node[above]{$4$};
        \draw (1,5) node[above]{$3$};
        \end{tikzpicture}
    \end{subfigure}
    \caption{4-cycle networks with four leaves. Under the CFN model these two networks have different ideals but the same matroid.}
    \label{fig:netL4C4}
\end{figure}
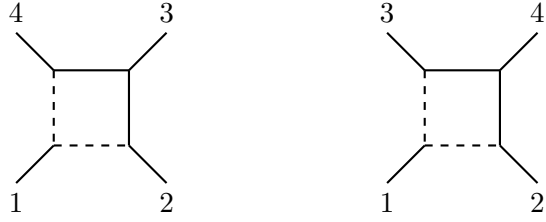

Group-based models on phylogenetic networks are quite similar to mixture models but with a restricted parameter space. The Fourier transform also applies to these network models so we let $G$ be either $\zz_2$ or $\zz_2 \times \zz_2$. Also let $N$ be a $k$-cycle network, and to each edge $e \in N$ and $g \in G$ associate a parameter $a_g^e$. Take $e_1$ and $e_2$ to be the reticulation edges and $T_1$ and $T_2$ to be the unrooted trees obtained from $N$ by removing either $e_1$ or $e_2$ respectively. Then the generic group based model associated to $G$ with network parameter $N$ is the image of the map
\[
\psi_N : \Theta_N \times [0,1] \to \Delta_{|G|^n - 1}
\]
given by
\[
\psi_N = \lambda \psi_{T_1} + (1-\lambda)\psi_{T_2}
\]
where $\psi_{T_i}$ is the map in the Fourier coordinates described in section \ref{sec:gpBasedModels} using the parameters that $T_i$ inherits from $N$. The main difference between these models and mixture models is that the parameters on each tree are not independent and actually overlap considerably. Similarly to the tree case, we let $V_N$ denote the Zariski closure of the image of $\psi_N$ and $I_N$ the vanishing ideal of $V_N$. 

\begin{ex}
Let $N$ be the network pictured in Figure \ref{fig:fourCyWithTrees}. The trees $T_1$ and $T_2$ that are also pictured in Figure \ref{fig:fourCyWithTrees} are obtained from $N$ by deleting the reticulation edges $e_5$ and $e_8$ and respectively. We denote the Fourier parameter corresponding to the edge $e_i$ and group element $g_j$ by $a_{g_j}^i$. The parameterization $\psi_N$ in the Fourier coordinates is
\[
q_{g_1, g_2, g_3, g_4} = \begin{cases}
a_{g_1}^1 a_{g_2}^2 a_{g_3}^3 a_{g_4}^4 a_{g_2}^6 a_{g_1 + g_4}^7 a_{g_1}^8 + 
a_{g_1}^1 a_{g_2}^2 a_{g_3}^3 a_{g_4}^4 a_{g_1}^5 a_{g_1 + g_2}^6 a_{g_4}^7
& \mbox{ if } \sum_{i \in [4]}g_i = 0 \\ 
0 & \mbox{ otherwise}
\end{cases}
\]
The first term in the above parameterization comes from the parameterization $\psi_{T_1}$ in the Fourier coordinates and the second term comes from $\psi_{T_2}$.
\end{ex}

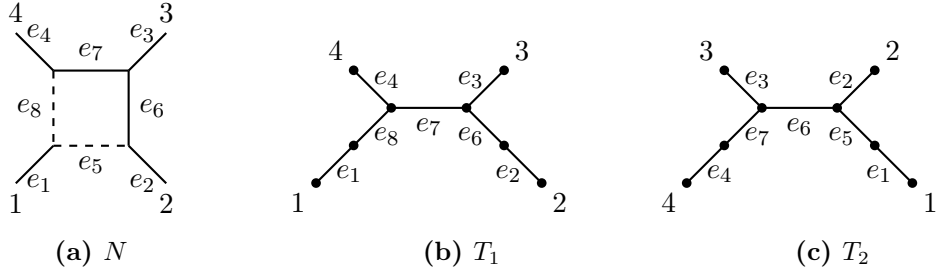
\begin{figure}
    \centering
    \begin{subfigure}[b]{0.3\linewidth}
        \centering
        \begin{tikzpicture}[scale = .5, thick]
        %\draw [help lines] (0,0) grid (6,6);
        \draw [dashed] (2,2)--(4,2);
        \draw (4,2)--(4,4);
        \draw (4,4)--(2,4);
        \draw [dashed] (2,4)--(2,2);
        
        \draw (2,2)--(1,1);
        \draw (4,2)--(5,1);
        \draw (4,4)--(5,5);
        \draw (2,4)--(1,5);
        
        \draw (1,1) node[below]{$1$};
        \draw (5,1) node[below]{$2$};
        \draw (5,5) node[above]{$3$};
        \draw (1,5) node[above]{$4$};
        
        \draw (1,1) node[right]{$e_1$};
        \draw (5,1) node[left]{$e_2$};
        \draw (5,5) node[left]{$e_3$};
        \draw (1,5) node[right]{$e_4$};
        
        \draw (3,2) node[below]{$e_5$};
        \draw (4,3) node[right]{$e_6$};
        \draw (3,4) node[above]{$e_7$};
        \draw (2,3) node[left]{$e_8$};
        \end{tikzpicture}
        \caption{$N$}
    \end{subfigure}
    \begin{subfigure}[b]{0.3\linewidth}
        \begin{tikzpicture}[scale = .5, thick]
        %\draw [help lines] (0,0) grid (6,4);
        
        \draw (0,0)--(2,2);
        \draw (2,2)--(4,2);
        \draw (4,2)--(6,0);
        \draw (2,2)--(1,3);
        \draw (4,2)--(5,3);
        
        \draw [fill] (0,0) circle [radius = .1];
        \draw [fill] (1,1) circle [radius = .1];
        \draw [fill] (2,2) circle [radius = .1];
        \draw [fill] (1,3) circle [radius = .1];
        \draw [fill] (4,2) circle [radius = .1];
        \draw [fill] (5,3) circle [radius = .1];
        \draw [fill] (5,1) circle [radius = .1];
        \draw [fill] (6,0) circle [radius = .1];
        
        \draw (0,0) node[below left] {$1$};
        \draw (1,3) node[above left] {$4$};
        \draw (5,3) node[above right] {$3$};
        \draw (6,0) node[below right] {$2$};
        
        \draw (.25,.25) node[right] {$e_1$};
        \draw (1.25,2.75) node[right] {$e_4$};
        \draw (4.75,2.75) node[left] {$e_3$};
        \draw (5.75,.25) node[left] {$e_2$};
        \draw (1.25,1.25) node[right] {$e_8$};
        \draw (3,2) node[below] {$e_7$};
        \draw (4.75,1.25) node[left] {$e_6$};
        \end{tikzpicture}
        \caption{$T_1$}
    \end{subfigure}
    \begin{subfigure}[b]{0.3\linewidth}
        \begin{tikzpicture}[scale = .5, thick]
        %\draw [help lines] (0,0) grid (6,4);
        
        \draw (0,0)--(2,2);
        \draw (2,2)--(4,2);
        \draw (4,2)--(6,0);
        \draw (2,2)--(1,3);
        \draw (4,2)--(5,3);
        
        \draw [fill] (0,0) circle [radius = .1];
        \draw [fill] (1,1) circle [radius = .1];
        \draw [fill] (2,2) circle [radius = .1];
        \draw [fill] (1,3) circle [radius = .1];
        \draw [fill] (4,2) circle [radius = .1];
        \draw [fill] (5,3) circle [radius = .1];
        \draw [fill] (5,1) circle [radius = .1];
        \draw [fill] (6,0) circle [radius = .1];
        
        \draw (0,0) node[below left] {$4$};
        \draw (1,3) node[above left] {$3$};
        \draw (5,3) node[above right] {$2$};
        \draw (6,0) node[below right] {$1$};
        
        \draw (.25,.25) node[right] {$e_4$};
        \draw (1.25,2.75) node[right] {$e_3$};
        \draw (4.75,2.75) node[left] {$e_2$};
        \draw (5.75,.25) node[left] {$e_1$};
        \draw (1.25,1.25) node[right] {$e_7$};
        \draw (3,2) node[below] {$e_6$};
        \draw (4.75,1.25) node[left] {$e_5$};
        \end{tikzpicture}
        \caption{$T_2$}
    \end{subfigure}
    \caption{A 4 leaf 4-cycle network $N$ and the two trees $T_1$ and $T_2$ that are obtained by deleting the reticulation edges $e_5$ and $e_8$ respectively.}
    \label{fig:fourCyWithTrees}
\end{figure}

It is not hard to see from the structure of the parameterization that if $T$ is one of the subtrees obtained from $N$ by deleting a reticulation edge of $N$, then $V_T \subseteq V_N$. This means that it will not be possible to separate tree models from other network models, and thus the set of cycle-networks will not be identifiable. This leads to study the following class of networks that was introduced in
\cite{phyloNetId18}.

\begin{defn}
The set of large-cycle networks is the collection of all $k$-cycle networks with $k \geq 4$.  
\end{defn}

\begin{defn}
The large-cycle network parameter of a phylogenetic network model is generically identifiable if for every pair of $n$-leaf large-cycle networks $N_1$ and $N_2$,
\[
\dim(V_{N_1} \cap V_{N_2}) < min(\dim(V_{N_1}), \dim(V_{N_2}))
\]
\end{defn}

\subsection{Identifiability of Network parameters for Large-Cycle K2P and K3P Networks}
In this section we describe the proof strategy that Gross and Long used to prove the generic identifiability of the network parameter for large-cycle JC networks. As they remarked in \cite{phyloNetId18}, the combinatorial arguments they make to prove the final result still apply but the necessary computational results are more difficult since K2P and K3P are higher dimensional models with more parameters. 

Let $M$ be a phylogenetic model for which the tree parameter is generically identifiable. Gross and Long showed in \cite[Section 4.2]{phyloNetId18} that if $M$ also satisfies the following three lemmas, then the large-cycle network parameter is identifiable for $M$. They prove this by finding subsets of the leaves of the networks that when restricted to, yield a situation that can be addressed with one of the lemmas or the generic identifiability of the tree parameter. In \cite{phyloNetId18}, they proved the same three results for the JC model by computing a degree-bounded Gr\"{o}bner basis for $I_N$ and then verifying that the degree-bounded basis generates a prime ideal of the correct dimension, thus it must be a Gr\"{o}bner basis for the prime ideal $I_N$. This computation becomes more difficult though as the number of parameters in the model increases. We instead use Algorithm \ref{alg:certID} to prove these lemmas for the K2P and K3P models. For the remainder of this paper we let $M$ be either K2P or K3P and denote the variety associated to the network $N$ under the model $M$ with $V_N^M$. 

\begin{lemma}
\label{lemma:cycleDim}
Let $N_1$ be a $k_1$-cycle network and $N_2$ be a $k_2$ cycle network. If $2 \leq k_1 < k_2 \leq 4$, then $V_{N_2}^M \not\subseteq V_{N_1}^M$.
\end{lemma}
\begin{proof}
We prove this by explicitly computing dimensions of the associated varieties. If $V_N^M$, is a network variety parameterized by $\psi_N^M$, then the dimension of $V_N^M$ can be computed by calculating the rank of the Jacobian of $\psi_N^M$ over the fraction field $k(\theta)$. In each case, we find that $\dim(V_{N_2}) > \dim(V_{N_1})$ which implies $V_{N_2}^M \not\subseteq V_{N_1}^M$. These computations can be found in the Mathematica files 
\emph{K2P\_Networks.nb} and \emph{K3P\_Networks.nb}. 
\end{proof}

\begin{lemma}
\label{lemma:fourCyID}
Let $N_1$ and $N_2$ be distinct 4-leaf 4-cycle networks. Then $V_{N_2}^M \not \subseteq V_{N_1}^M$ and $V_{N_1}^M \not \subseteq V_{N_2}^M$.
\end{lemma}
\begin{proof}
In this case $V_{N_1}$ and $V_{N_2}$ both have the same dimension so we can run the specialized version of Algorithm \ref{alg:certID}. For both models we ran \emph{matroidSeparate} and were once again able to find a certificate separating each pair of 4-leaf 4-cycle networks.
These computations can also be found in the Mathematica files \emph{K2P\_Networks.nb} and \emph{K3P\_Networks.nb}.
\end{proof}

\begin{lemma}
\label{lemma:sunletID}
Let $N_1$ be either of the two 5-leaf 4-cycle networks pictured in Figure \ref{fig:netL4C5} and let $N_2$ be the 5-leaf 5-cycle network with reticulation edges directed toward the leaf-labelled by 1. Then $V_{N_1} \not \subseteq V_{N_2}$.
\end{lemma}
\begin{proof}
$V_{N_1}$ and $V_{N_2}$ once again have the same dimension in this case so we can again run the specialized version of Algorithm \ref{alg:certID} to show $V_{N_1} \not \subseteq V_{N_2}$ for both possible choices of $N_1$. As before, we ran \emph{matroidSeparate} to find certificates that show $V_{N_1} \not \subseteq V_{N_2}$.  These computations can also be found in the Mathematica files \emph{K2P\_Networks.nb} and \emph{K3P\_Networks.nb}.
\end{proof}

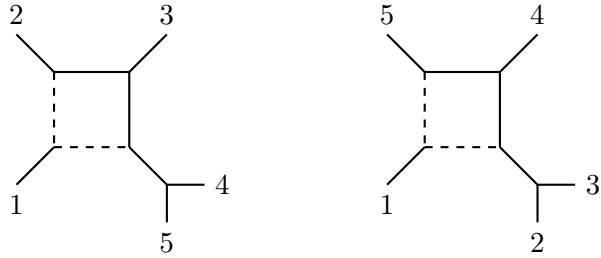
\begin{figure}
    \begin{subfigure}[b]{0.3\linewidth}
        \centering
        \begin{tikzpicture}[scale = .5, thick]
        %\draw [help lines] (0,0) grid (7,7);
        \draw [dashed] (2,2)--(4,2);
        \draw (4,2)--(4,4);
        \draw (4,4)--(2,4);
        \draw [dashed] (2,4)--(2,2);
        
        \draw (2,2)--(1,1);
        \draw (4,2)--(5,1);
        \draw (4,4)--(5,5);
        \draw (2,4)--(1,5);
        \draw (5,1)--(5,0);
        \draw (5,1)--(6,1);
        
        \draw (1,1) node[below]{$1$};
        \draw (5,5) node[above]{$3$};
        \draw (1,5) node[above]{$2$};
        \draw (5,0) node[below]{$5$};
        \draw (6,1) node[right]{$4$};
        
        \end{tikzpicture}
    \end{subfigure}
    \begin{subfigure}[b]{0.3\linewidth}
        \centering
        \begin{tikzpicture}[scale = .5, thick]
        %\draw [help lines] (0,0) grid (7,7);
        \draw [dashed] (2,2)--(4,2);
        \draw (4,2)--(4,4);
        \draw (4,4)--(2,4);
        \draw [dashed] (2,4)--(2,2);
        
        \draw (2,2)--(1,1);
        \draw (4,2)--(5,1);
        \draw (4,4)--(5,5);
        \draw (2,4)--(1,5);
        \draw (5,1)--(5,0);
        \draw (5,1)--(6,1);
        
        \draw (1,1) node[below]{$1$};
        \draw (5,5) node[above]{$4$};
        \draw (1,5) node[above]{$5$};
        \draw (5,0) node[below]{$2$};
        \draw (6,1) node[right]{$3$};
        
        \end{tikzpicture}
    \end{subfigure}
    \caption{The two possibilities for $N_1$ in Lemma \ref{lemma:sunletID}.}
    \label{fig:netL4C5}
\end{figure}

\begin{cor}
The semi-directed network parameter of large-cycle K2P and K3P network models is generically identifiable.
\end{cor}
\begin{proof}
Since Lemmas \ref{lemma:cycleDim}, \ref{lemma:fourCyID}, \ref{lemma:sunletID} hold for K2P and K3P cycle-networks, Lemmas 4.11, 4.12, and 4.13 of \cite{phyloNetId18} hold for K2P and K3P networks as well. This means for any two large-cycle networks $N_1$ and $N_2$, $V_{N_1} \not \subseteq V_{N_2}$ and 
$V_{N_2} \not \subseteq V_{N_1}$. Since these varieties are irreducible, this mutual non-containment implies
\[
\dim(V_{N_1}^M \cap V_{N_2}^M) < \min(\dim(V_{N_1}^M), \dim(V_{N_2}^M))
\]
and so the semi-directed network parameter of large-cycle K2P and K3P network models is generically identifiable. 
\end{proof}

\begin{rmk}
In our original computations we were also able to separate the 3-cycle networks from the 4-cycle networks for both the K2P and K3P models. It may be possible to extend these identifiability results to cycle networks with cycle size at least 3. As previously mentioned though, it will always be impossible for trees to be generically identifiable from cycle networks. 
\end{rmk}

This serves as another example of how Algorithm \ref{alg:certID} can be used to obtain identifiability results for discrete parameters in algebraic models. While this algorithm has nice computational advantages over computing vanishing ideals, there can be times when it fails to separate varieties whose intersection is actually lower dimensional. It is important to remember that when this algorithm fails to separate two models, it does not imply that the discrete parameter is not identifiable. The example below shows that even if we compute the entire matroid of both models, we still may not be able to separate models whose intersection is actually lower dimensional.

\begin{ex}
\label{ex:matroidIDFail}
Let $N_1$ and $N_2$ be the networks pictured on the left and right in Figure \ref{fig:netL4C4} respectively. We can directly compute the vanishing ideals $I_{N_1}$ and $I_{N_2}$ of the CFN network models on $N_1$ and $N_2$ via elimination and get
\begin{align*}
I_{N_1} &= \langle q_{0110}q_{1001} - q_{0101}q_{1010} + q_{0011}q_{1100} - q_{0000}q_{1111} \rangle \\
I_{N_2} &= \langle -q_{0110}q_{1001} + q_{0101}q_{1010} + q_{0011}q_{1100} - q_{0000}q_{1111} \rangle.
\end{align*}
These ideals are of the same dimension and not equal so the intersection of their corresponding varieties is lower dimensional. Despite that, we can compute their entire matroid explicitly and see that they are equal. This stems from the fact that the polynomials that generate $I_{N_1}$ and $I_{N_2}$ involve the same variables. 
\end{ex}

%%%%%%%%%%%%%%%%%%%%%%%%%%%%%%%%%%%%%%%%%%%%%%%%%%%%%%%%%%%%%
%%%%%%%%%%%%%%%%%%%%%%%%%%%%%%%%%%%%%%%%%%%%%%%%%%%%%%%%%%%%%
%%%%%%%%%%%%%%%%%%%%%%%%%%%%%%%%%%%%%%%%%%%%%%%%%%%%%%%%%%%%%
%%%%%%%%%%%%%%%%%%%%%%%%%%%%%%%%%%%%%%%%%%%%%%%%%%%%%%%%%%%%%

\section*{Acknowledgments}

Benjamin Hollering and Seth Sullivant were partially supported by the US National Science Foundation
(DMS 1615660).

\nocite{*}
\bibliography{phylo_id_and_alg_matroids.bib}{}
\bibliographystyle{plain}

\end{document}